\documentclass[12pt,reqno]{amsart}
\usepackage{fullpage}
\usepackage{times}
\usepackage{amsmath,amssymb,amsthm,url}
\usepackage[utf8]{inputenc}
\usepackage[english]{babel}
\usepackage{comment}
\usepackage{bbm}
\usepackage{enumerate}
\usepackage{bm}
\usepackage{graphicx}
\usepackage{mathrsfs}
\usepackage[colorlinks=true, pdfstartview=FitH, linkcolor=blue, citecolor=blue, urlcolor=blue]{hyperref}
\usepackage[vlined,ruled]{algorithm2e}

\newtheorem{thm}{Theorem}[section]
\newtheorem{lem}[thm]{Lemma}

\newtheorem{cor}[thm]{Corollary}

\theoremstyle{definition}

\renewcommand{\pmod}[1]{{\ifmmode\text{\rm\ (mod~$#1$)}\else\discretionary{}{}{\hbox{ }}\rm(mod~$#1$)\fi}}

\title{Improved upper bounds on Diophantine tuples with the property $D(n)$}
\author{Chi Hoi Yip}
\address{Department of Mathematics \\ University of British Columbia \\ Vancouver  V6T 1Z2 \\ Canada}
\email{kyleyip111@gmail.com}
\begin{document}

\begin{abstract}
Let $n$ be a non-zero integer. A set $S$ of positive integers is a Diophantine tuple with the property $D(n)$ if $ab+n$ is a perfect square for each $a,b \in S$ with $a \neq b$. It is of special interest to estimate the quantity $M_n$, the maximum size of a Diophantine tuple with the property $D(n)$. In this notes, we show the contribution of intermediate elements is $O(\log \log |n|)$, improving a result by Dujella. As a consequence, we deduce that $M_n\leq (2+o(1))\log |n|$, improving the best-known upper bound on $M_n$ by Becker and Murty.
\end{abstract}

\maketitle

\section{Introduction}

A set $\{a_{1},a_{2},\ldots, a_{m}\}$ of distinct positive integers is a \textit{Diophantine $m$-tuple} if the product of any two distinct elements in the set is one less than a square. A famous example of a Diophantine quadruple is $\{1,3,8,120\}$, due to Fermat. Such a construction is optimal in the sense that there is no Diophantine $5$-tuple, recently confirmed by He, Togb\'e, and Ziegler \cite{HTZ19}. There are many generalizations and variants of Diophantine tuples. We refer to the recent book of Dujella \cite{D24} for a comprehensive overview of the topic. 

In this paper, we focus on one natural generalization that has been studied extensively. Let $n$ be a non-zero integer. A set $S$ of positive integers is a Diophantine tuple with the property $D(n)$ if $ab+n$ is a perfect square for each $a,b \in S$ with $a \neq b$. It is of special interest to estimate the quantity $M_n$, the maximum size of a Diophantine tuple with the property $D(n)$. We have mentioned that $M_1=4$ \cite{HTZ19}. Analogously, Bliznac Trebje\v sanin and Filipin \cite{BTF19} proved that $M_4=4$. More recently, Bonciocat, Cipu, and Mignotte \cite{BCM22} proved that $M_{-1}=M_{-4}=3$. 

It is widely believed that $M_n$ is uniformly bounded (for example, this follows from the uniformity conjecture) \cite{BM19, D02}. Using elementary congruence consideration, it is easy to show that $M_n=3$ when $n \equiv 2 \pmod 4$; see \cite[Section 5.4.1]{D24} for more discussions. On the other hand, in a remarkable paper \cite{DL05}, Dujella and Luca showed that if $p$ is a prime, then $M_p$ and $M_{-p}$ are both bounded by $3 \cdot 2^{168}$. However, for a generic $n$, the best-known upper bound on $M_n$ is of the form $O(\log |n|)$.

Following \cite{BM19, D02, D04}, for Diophantine tuples with the property $D(n)$, we separate the contribution of large, intermediate, and small elements as follows:
\begin{align*}
& A_n=\sup \left\{\left|S \cap\left[|n|^3,+\infty\right)\right|: S \text { has the property } D(n)\right\}, \\
& B_n=\sup \left\{\left|S \cap(n^2,|n|^3)\right|: S \text { has the property } D(n)\right\}, \\
& C_n=\sup \left\{\left|S \cap[1, n^2]\right|: S \text { has the property } D(n)\right\}.
\end{align*}
In \cite{D02}, Dujella showed that $A_n \leq 31$. The best-known upper bound on $B_n$ is $B_n \leq  0.6071 \log |n|+O(1)$, due to Dujella \cite{D04}. As for $C_n$, the best record 
\begin{equation}\label{eq:BM}
C_n\leq 2\log |n|+O\bigg(\frac{\log |n|}{(\log \log |n|)^2}\bigg)    
\end{equation}
 is due to Becker and Murty \cite{BM19}. Summing the bounds on $A_n, B_n$, and $C_n$ yields $M_n\leq (2.6071+o(1))\log |n|$, the best-known upper bound on $M_n$ \cite{BM19}.

Our main result is the following improved upper bound on $B_n$ and $M_n$. 

\begin{thm}\label{thm:main}
$$B_n=O(\log \log |n|), \quad M_n\leq 2\log |n|+O\bigg(\frac{\log |n|}{(\log \log |n|)^2}\bigg).$$
\end{thm}

The key observation of our improvement is that the contribution of intermediate elements can be bounded more efficiently. More precisely, to achieve that, we separate the contribution of large and intermediate elements differently. For each $\epsilon>0$, let
\begin{align*}
&  A_n^{(\epsilon)}=\sup \left\{\left|S \cap(|n|^{2+\epsilon},+\infty)\right|: S \text { has the property } D(n)\right\},\\
&  B_n^{(\epsilon)}=\sup \left\{\left|S \cap(n^2, |n|^{2+\epsilon}]\right|: S \text { has the property } D(n)\right\}.
\end{align*}

We give the following estimate on $A_n^{(\epsilon)}$ and $B_n^{(\epsilon)}$.

\begin{thm}\label{thm:main2}
Thw following estimates hold uniformly for all $\epsilon \in (0,1)$ and all non-zero integers $n$:
$$
A_{n}^{(\epsilon)}=O\bigg(\log \frac{1}{\epsilon}\bigg), \quad  B_n^{(\epsilon)}\leq 0.631\epsilon\log |n|+O(1).
$$
\end{thm}

To estimate $B_n$ and $M_n$, note that
$$
B_n \leq A_{n}^{(\epsilon)}+B_n^{(\epsilon)}, \quad M_n\leq A_{n}^{(\epsilon)}+B_n^{(\epsilon)}+C_n.
$$
Setting 
$$\epsilon=\frac{\log \log |n|}{\log |n|},$$ 
Theorem~\ref{thm:main} follows from Theorem~\ref{thm:main2} and inequality~\eqref{eq:BM} immediately. 

\section{Proofs}
Our proofs are inspired by several arguments used in \cite{D02, D04, Y24}.

We first recall three useful lemmas from \cite{D02}.

\begin{lem}[{\cite[Lemma 2]{D02}}]\label{lem:bound}
Let $n$ be a nonzero integer. Let $\{a, b, c, d \}$ be a Diophantine quadruple with the property $D(n)$ and $a<b<c<d$. If $c>b^{11}|n|^{11}$, then $d\leq c^{131}$.    
\end{lem}

\begin{lem}[{\cite[Lemma 3]{D02}}]\label{lem:e}
Let $n$ be a nonzero integer. If $\{a, b, c\}$ is a Diophantine triple with the property $D(n)$ and $ab+n=r^2, ac+n=s^2, bc+n=t^2$, then there exist integers $e, x, y, z$ such that
$$
a e+n^2=x^2, \quad b e+n^2=y^2, \quad c e+n^2=z^2,
$$
and
$$
c=a+b+\frac{e}{n}+\frac{2}{n^2}(abe+rxy).
$$    
\end{lem} 

\begin{lem}[{\cite[Lemma 5]{D02}}]\label{lem:cd}
Let $n$ be an integer with $|n|\geq 2$. Let $\{a, b, c, d \}$ be a Diophantine quadruple with the property $D(n)$. If $n^2< a<b<c<d$, then $c>3.88a$ and $d>4.89c$.    
\end{lem}

Next, we deduce a gap principle from the above two lemmas. 

\begin{cor}\label{cor:gap}
Let $n$ be an integer with $|n|\geq 2$. Let $\{a, b, c, d \}$ be a Diophantine quadruple with the property $D(n)$. If $n^2< a<b<c<d$, then $$d>\frac{bc}{n^2}.$$  
\end{cor}
\begin{proof}
We apply Lemma~\ref{lem:e} to the Diophantine triple $\{b,c,d\}$. Since $b>n^2$ and $be+n^2 \geq 0$, it follows that $e \geq 0$. If $e=0$, then Lemma~\ref{lem:e} implies that
$$
d=b+c+2\sqrt{bc+n}<2c+2\lfloor\sqrt{c^2+n}\rfloor\leq 4c<4.89c,
$$
which is impossible in view of Lemma~\ref{lem:cd}. Thus, $e\geq 1$ and Lemma~\ref{lem:e} implies that
$$
d>\frac{2bce}{n^2}>\frac{bc}{n^2}.
$$  
\end{proof}

Now we are ready to prove Theorem~\ref{thm:main2}.

\begin{proof}[Proof of Theorem~\ref{thm:main2}]
Let $\epsilon \in (0,1)$ and $n$ be a nonzero integer. 

We first bound $B_n^{(\epsilon)}$. Let $S$ be a Diophantine tuple with the property $D(n)$, such that all elements in $S$ are in $[n^2, |n|^{2+\epsilon}]$. By Lemma~\ref{lem:cd}, the elements in $S$ grows exponentially and more precisely, we have 
\begin{equation}\label{eq:difference}
|S| \leq \epsilon \log_{4.89} |n| +O(1)<0.631\epsilon \log |n|+O(1).    
\end{equation}
The bound on $B_n^{(\epsilon)}$ follows. 

Next we estimate $A_n^{(\epsilon)}$. Let $S$ be a Diophantine tuple with the property $D(n)$, such that all elements in $S$ are at least $|n|^{2+\epsilon}$. Label the elements in $S$ in increasing order as $a_1<a_2<\cdots$. By Corollary~\ref{cor:gap}, 
$$
a_{i+2}>\frac{a_{i}a_{i+1}}{n^2}
$$
holds for each $i \geq 2$. For each $i \geq 2$, let $b_i=a_i/n^2$. Then we have $b_{i+2}>b_ib_{i+1}$. Note that $b_2=a_2/n^2$ and $b_3>b_2=a_2/n^2$. Define the sequence $\{\beta_i\}_{i=2}^{\infty}$ recursively by the following:
$$
\beta_2=\beta_3=1, \quad \beta_{i+2}=\beta_i+\beta_{i+1} (i \geq 2). $$
By induction, we have $b_i>(a_2/n^2)^{\beta_i}$. It follows that
\begin{equation}\label{eq:a_i}
a_i>\frac{a_2^{\beta_i}}{|n|^{2\beta_i-2}}.    
\end{equation}
Since $\beta_i \to \infty$, we can choose $k$ sufficiently large such that 
$$(\beta_k-11)(2+\epsilon)>2\beta_k+9;
$$
let $k=k(\epsilon)$ be the smallest such $k$.
If $|S|<k$, we are done. Otherwise, inequality~\eqref{eq:a_i} implies that  
$$
a_k>\frac{a_2^{\beta_k}}{|n|^{2\beta_k-2}}=a_2^{11}|n|^{11} \cdot \frac{a_2^{\beta_k-11}}{|n|^{2\beta_k+9}}> a_2^{11}|n|^{11} |n|^{(\beta_k-11)(2+\epsilon)-(2\beta_k+9)}> a_2^{11}|n|^{11}.
$$

Now Lemma~\ref{lem:bound}
implies that the largest element in $S$ is at most $a_k^{131}$. By a similar argument as above, for each $i \geq 2$, we have
\begin{equation}\label{eq:eq2}
a_{k+i}> \frac{a_k^{\beta_i}}{|n|^{2\beta_i-2}}.
\end{equation}
Since $\beta_i \to \infty$, we can choose $\ell$ sufficiently large such that 
$$(\beta_{\ell}-131)(2+\epsilon)>2\beta_{\ell}-2;
$$
let $\ell=\ell(\epsilon)$ be the smallest such $\ell$. Note that both $k$ and $\ell$ are explicitly computable constants depending only on $\epsilon$. Since the sequence $(\beta_i)$ grows exponentially, it follows that $k(\epsilon)$ and $\ell(\epsilon)$ are of the order $\log \frac{1}{\epsilon}$. 

If $|S|\geq k+\ell$, then inequality~\eqref{eq:eq2} implies that
$$
a_{k+\ell}> \frac{a_k^{\beta_{\ell}}}{|n|^{2\beta_{\ell}-2}}=a_k^{131} \cdot \frac{a_k^{\beta_{\ell}-131}}{|n|^{2\beta_{\ell}-2}}>a_k^{131} |n|^{(\beta_{\ell}-131)(2+\epsilon)-(2\beta_{\ell}-2)}>a_k^{131},
$$
a contradiction. Therefore, $|S|< k+\ell$. Thus, $A_n^{(\epsilon)}\leq k(\epsilon)+\ell(\epsilon)\ll \log \frac{1}{\epsilon}$, where the implicit constant is absolute. 
\end{proof}

\section*{Acknowledgement}
The author thanks anonymous referees for their valuable comments and suggestions.

\bibliographystyle{abbrv}
\bibliography{references}

\end{document}